\documentclass{my-aims}

\usepackage{amsmath}
\usepackage{paralist}
\usepackage[colorlinks=true]{hyperref}

\hypersetup{urlcolor=blue, citecolor=red}

\def\ppages{X--XX}

\newtheorem{theorem}{Theorem}[section]
\newtheorem{corollary}{Corollary}

\theoremstyle{definition}
\newtheorem{definition}[theorem]{Definition}
\newtheorem{remark}{Remark}

\newtheorem*{problem}{Problem $\mathbf{(H_\tau^n)}$}

\title[Noether currents for problems of Herglotz]{%
Noether currents for higher-order variational problems of Herglotz type with time delay}

\author[S. P. S. Santos, N. Martins and D. F. M. Torres]{Sim\~{a}o P. S. Santos,
Nat\'{a}lia Martins and Delfim F. M. Torres}

\address[Sim\~{a}o P. S. Santos, Nat\'{a}lia Martins and Delfim F. M. Torres]{Center
for Research and Development in Mathematics and Applications (CIDMA),
Department of Mathematics, University of Aveiro, 3810-193 Aveiro, Portugal}
\email{{\tt spsantos@ua.pt, natalia@ua.pt, delfim@ua.pt}}



\begin{document}


\begin{abstract}
We study, from an optimal control perspective, 
Noether currents for higher-order problems 
of Herglotz type with time delay.
Main result provides new Noether currents for such 
generalized variational problems, 
which are particularly useful in the search of extremals. 
The proof is based on the idea of rewriting the higher-order 
delayed generalized variational problem as a first-order 
optimal control problem without time delays.
\end{abstract}

\maketitle


\section{Introduction}

This article is devoted to the proof of a second Noether 
type theorem for higher-order delayed variational problems of Herglotz.
Such problems, which are invariant under a certain group of transformations,
were first studied in 1918 by Emmy Noether  
for the particular case of first-order variational 
problems without time delay \cite{Noether1918}. 
In her famous paper \cite{Noether1918}, Noether proved 
two remarkable theorems that relate the invariance 
of a variational integral with properties 
of its Euler--Lagrange equations. Since most
physical systems can be described by using Lagrangians 
and their associated actions, the importance 
of Noether's two theorems is obvious \cite{MR3540628}.

The first Noether's theorem, usually simply called Noether's theorem,
ensures the existence of $r$ conserved quantities along 
the Euler--Lagrange extremals when the variational integral 
is invariant with respect to a continuous symmetry 
transformation that depend on $r$ parameters \cite{cpaa}.
Noether's theorem explains all conservation laws of mechanics,  
for instance, invariance under translation in time 
implies conservation of energy; 
conservation of linear momentum comes from invariance of the system 
under spacial translations; invariance under rotations 
in the base space yields conservation of angular momentum.

The second Noether's theorem, less known than the first one,
applies to variational problems that are invariant under 
a certain group of transformations that depends on arbitrary 
functions and their derivatives up to some order \cite{Torres2003MR1980565}. 
In contrast to Noether's theorem, where the transformations are global,
in second Noether's theorem the transformations are local: 
they can affect every part of the system differently.
Noether's second theorem has  applications in several fields, 
such as, general relativity, hydromechanics, electrodynamics,
and quantum chromodynamics \cite{MR3413358,MR2761345,MR3467590}. 
Extensions of both Noether's theorems to optimal control problems were 
first obtained in \cite{ejc,Torres:ConservLaws2002,Torres2003MR1980565,Torres2004}. 
For systems with time delay, see \cite{MR2970905}.
In 2013, the second Noether theorem was extended to the context of 
fractional calculus \cite{MR3179312} and time scales \cite{malinaNaty}.

Motivated by the important applications of Noether's second theorem \cite{malinaNaty}
and the applicability of higher-order dynamic systems with time delays
in modeling real-life phenomena \cite{MyID:304,MyID:253,MyID:353}, as well as the 
importance of variational problems of Herglotz \cite{Guenther1996,Herglotz1930},  
our goal in this paper is to study generalized variational problems 
that are invariant under a certain group of transformations 
that depends on arbitrary functions and their derivatives up to some order, 
and deduce expressions for Noether currents, that is, 
expressions that are constant in time along the extremals.

Our work is related with the second Noether theorem for optimal control
in the sense of \cite{Torres2003MR1980565}, and is particularly useful 
because provides necessary conditions for the search of extremals.
There are other different results on the calculus of variations, 
also related with the notion of invariance under a certain group 
of transformations that depends on arbitrary functions 
and their derivatives, but they are concerned with Noether identities 
\cite{Georgieva2005,malinaNaty,Agnieszka+Tatiana2016} 
and not with Noether currents as we do here.

The generalized variational problem was introduced 
by Herglotz in 1930 \cite{Herglotz1930},
and consists in the determination of
$x\in C^1([a,b];\mathbb{R}^m)$
and $z\in C^1([a,b];\mathbb{R})$,
such that
\begin{equation}
\label{PH}
\tag{$H^1$}
\begin{gathered}
z(b)\longrightarrow \mathrm{extr},\\
\dot{z}(t)=L(t,x(t),\dot{x}(t),z(t)), \quad t \in [a,b],\\
\text{subject to } x(a)=\alpha \quad \text{and} \quad z(a)=\gamma
\end{gathered}
\end{equation}
for some $\alpha \in \mathbb{R}^m$ and $\gamma \in \mathbb{R}$,
where by ``$\mathrm{extr}$'' we mean ``to minimize or maximize'' and
the Lagrangian $L \in C^1([a,b]\times\mathbb{R}^{2m+1}; \mathbb{R})$ 
is such that $\displaystyle t \mapsto \frac{\partial L}{\partial x}
\left(t,x(t), \dot{x}(t), z(t)\right)$, $\displaystyle t
\mapsto \frac{\partial L}{\partial \dot{x}}\left(t,
x(t), \dot{x}(t), z(t)\right)$  and $\displaystyle t \mapsto
\frac{\partial L}{\partial z}\left(t, x(t), \dot{x}(t), z(t)\right)$
are differentiable. It is clear that if the Lagrangian 
$L$ does not depend on the variable $z$,
then we get the classical problem of the calculus of variations.
The variational problem of Herglotz attracted the interest of the mathematical
community in the last two decades, after the publications 
\cite{Guenther1996SIAM,Guenther1996}. Namely, the two Noether theorems were proved 
for the first-order problem in \cite{Georgieva2002,Georgieva2005,Georgieva2003}. 
The first Noether theorem for variational problems of Herglotz type with time delay 
was proved in \cite{MyArt02}. The higher-order problem of Herglotz was introduzed in \cite{MyArt01}.
Noether's first theorem for higher-order problems was proved in \cite{MyArt04} and, 
more recently, using an optimal control approach, the authors generalized previous 
results for higher-order problems with time delay in \cite{MyArt05}.
The variational problem of Herglotz was also considered in the context 
of fractional calculus in \cite{Almeida+Malinowska2014} and, in the general 
context of Riemannian manifolds, in \cite{Ligia+Luis+Natalia2016}.

The manuscript is organized as follows. In Section~\ref{sec:prelim}, we present 
the results that constitute the basis of our work: a version of 
Pontryagin's maximum principle, higher-order delayed Euler--Lagrange equations
and Noether's second theorem for optimal control problems. In Section \ref{sec:MainRes}, 
we prove our main results: a second Noether theorem for higher--order problems 
of Herglotz with time delay (Theorem~\ref{thm 2_Noether}) 
and two important corollaries: the first (Corollary~\ref{coroll:n=1}) 
is devoted to first-order variational  problems of Herglotz 
with time delay, while the second (Corollary~\ref{coroll:2}) 
is devoted to first-order classical variational problems with time delay.
We finish the paper with an illustrative example (Section~\ref{sec:ex})
and concluding remarks (Section~\ref{sec:conc}).


\section{Preliminaries}
\label{sec:prelim}

In this paper we consider the following generalized 
variational problem $\mathbf{(H_\tau^n)}$.
\begin{problem}
\textit{Let $\tau$ be a real number such that $0\leq \tau<b-a$.
Determine piecewise trajectories $x\in PC^n([a-\tau,b];\mathbb{R}^m)$
and a function $z\in PC^1([a,b];\mathbb{R})$ such that:}
\begin{equation*}
z(b)\longrightarrow \mathrm{extr},\\
\end{equation*}
\textit{where the pair} $(x(\cdot),z(\cdot))$
\textit{satisfies the differential equation}
\begin{equation*}
\dot{z}(t)=L\left(t,x(t),\dot{x}(t),\dots, x^{(n)}(t),
x(t-\tau),\dot{x}(t-\tau),\dots, x^{(n)}(t-\tau),z(t)\right),
\end{equation*}
\textit{for} $t \in [a,b]$, \textit{and is subject to initial conditions}
\begin{equation*}
z(a)=\gamma \in \mathbb{R} \quad  \text{and} \quad 
x^{(k)}(t)=\mu^{(k)}(t), \quad k=0,\dots,n-1,
\end{equation*}
\textit{where} $\mu \in PC^n ([a-\tau,a];\mathbb{R}^m)$
\textit{ is a given initial function. The Lagrangian
$L$ is assumed to satisfy the following hypotheses:}
\begin{itemize}
\item[i.] $L \in C^1([a,b]\times\mathbb{R}^{2m(n+1)}; \mathbb{R})$;

\item[ii.] \textit{functions} $t \mapsto
\frac{\partial L}{\partial z}[x;z]_\tau^n(t)$,
$t\mapsto \frac{\partial L}{\partial x^{(k)}}[x;z]_\tau^n(t)$
\textit{and} $t\mapsto \frac{\partial L}{\partial x_\tau^{(k)}}[x;z]_\tau^n(t)$
\textit{are differentiable for any admissible pair}
$(x(\cdot),z(\cdot))$, $k=0,\dots,n$,
\end{itemize}
\textit{where, to simplify expressions, we use the notation
$x_\tau^{(k)}(t)$ to denote the $k$th derivative
of $x$ evaluated at $t-\tau$ (often we use
$x_\tau(t)$ for $x_\tau^{(0)}(t) = x(t-\tau)$ and
$\dot{x}_\tau(t)$ for $x_\tau^{(1)}(t) = \dot{x}(t-\tau)$) and
$$
[x;z]^n_\tau(t):=\left(t,x(t),\dot{x}(t),
\dots, x^{(n)}(t), x_\tau(t),\dot{x}_\tau(t),\dots, x_\tau^{(n)}(t),z(t)\right).
$$}
\end{problem}

Associated with the generalized variational problem $\mathbf{(H_\tau^n)}$,
one has the following definitions.

\begin{definition}[Admissible pair to problem $\mathbf{(H_\tau^n)}$]
We say that $(x(\cdot),z(\cdot))$ with $x(\cdot) \in PC^n([a-\tau,b];\mathbb{R}^m)$
and $z(\cdot) \in PC^1([a,b];\mathbb{R})$ is an admissible pair to problem
$\mathbf{(H_\tau^n)}$ if it satisfies the equation
$$
\dot{z}(t)=L[x;z]^n_\tau(t), \quad t \in [a,b],
$$
subject to
$$
z(a)=\gamma \quad \mbox{and} \quad x^{(k)}(t)=\mu^{(k)}(t)
$$
for all $k=0,1,\dots,n-1$, $t\in [a-\tau,a]$
and $\gamma \in \mathbb{R}$.
\end{definition}

\begin{definition}[Extremizer to problem $\mathbf{(H_\tau^n)}$]
An admissible pair $(x^*(\cdot),z^*(\cdot))$ is said to be an extremizer
to problem $\mathbf{(H_\tau^n)}$ if $z(b)-z^*(b)$ has the same signal for all
admissible pairs $(x(\cdot),z(\cdot))$ that satisfy
$\|z-z^* \|_0< \epsilon$ and $\|x-x^* \|_0< \epsilon$
for some positive real $\epsilon$, where
$\|y\|_0=\smash{\displaystyle\max_{a\leq t \leq b}}|y(t)|$.
\end{definition}

Inspired by the ideias presented in \cite{Guinn}
(see also \cite{MR3177826,MR3124697,MR3531794,MyArt05}), 
problem $\mathbf{(H_\tau^n)}$ can be rewritten as a first-order optimal 
control problem without time delay. Such reduction is presented in 
Section~\ref{sec:MainRes}. Firstly, let us recall some key notions 
and results from optimal control theory.
Consider the optimal control problem in Bolza form  
on the interval $[a,b]$:
\begin{equation}
\label{problem P}
\tag{$P$}
\begin{gathered}
\mathcal{J}(x(\cdot),u(\cdot))=\int_a^b f(t,x(t),u(t))dt
+\phi(x(b))\longrightarrow \mathrm{extr}\\
\text{subject to } \dot{x}(t)=\varphi(t,x(t),u(t)),
\end{gathered}
\end{equation}
with some initial condition on $x$,
where $f \in C^1([a,b]\times \mathbb{R}^{m}\times \Omega;\mathbb{R})$,
$\phi \in C^1(\mathbb{R}^{m};\mathbb{R})$,
$\varphi \in C^1([a,b]\times \mathbb{R}^{m}\times \Omega;\mathbb{R}^m)$,
$x \in PC^1([a,b]; \mathbb{R}^m)$ and $u\in PC([a,b];\Omega)$,
with $\Omega \subseteq \mathbb{R}^r$ an open set.
Function $x$ is called the state variable and $u$ the control variable; 
$\phi$ is known as the payoff term. 

A fundamental tool in optimal control theory 
is the well-known Ponytryagin's maximum principle.

\begin{theorem}[Pontryagin's maximum principle
for problem \eqref{problem P} \cite{Pontryagin}]
\label{PMP}
If a pair $(x(\cdot),u(\cdot))$ with
$x \in PC^1([a,b]; \mathbb{R}^m)$ and $u\in PC([a,b];\Omega)$
is a solution to problem \eqref{problem P} with the initial condition $x(a)=\alpha$,
$\alpha \in \mathbb{R}^m$, then there exists a multiplier $\psi \in PC^1([a,b];\mathbb{R}^m)$
such that for the Hamiltonian $H$ defined by
\begin{equation}
\label{eq:def:Hamiltonian}
H(t,x,u,\psi):=f(t,x,u)+\psi\cdot \varphi(t,x,u)
\end{equation}
the next conditions hold:
\begin{itemize}
\item the optimality condition
\begin{equation}
\label{prob P opt condt}
\frac{\partial H}{\partial u}(t, x(t),u(t), \psi(t))=0;
\end{equation}
		
\item the adjoint system
\begin{equation}
\label{prob P adj syst}
\begin{cases}
\dot{x}(t)=\frac{\partial H}{\partial \psi}(t, x(t),u(t), \psi(t))\\
\dot{\psi}(t)=-\frac{\partial H}{\partial x}(t, x(t),u(t), \psi(t));
\end{cases}
\end{equation}
		
\item the transversality condition
\begin{equation}
\label{prob P tr condt}
\psi(b)=grad(\phi(x))(b).
\end{equation}
\end{itemize}
\end{theorem}

The following definition is of central importance 
for the formulation of second Noether's theorem.

\begin{definition}[Noether current \cite{Torres2003MR1980565}]
A function $C(t,x(t),u(t), \psi(t))$, which is constant along 
every $x \in PC^1([a,b];\mathbb{R}^m)$, $u\in PC([a,b];\Omega)$ and
$\psi \in PC^1([a,b];\mathbb{R}^m)$ solution of 
\eqref{prob P opt condt}--\eqref{prob P tr condt},
is called a Noether current.
\end{definition}

In order to apply the results of \cite{Torres2003MR1980565}
to the Bolza problem \eqref{problem P}, we rewrite
it in the following equivalent Lagrange form:
\begin{equation*}
\begin{gathered}
\mathcal{I}(x(\cdot),y(\cdot),u(\cdot))
=\int_a^b \big(f(t,x(t),u(t))+ y(t)\big) dt \longrightarrow \mathrm{extr}\\
\text{subject to }
\begin{cases}
\dot{x}(t)=\varphi\left(t,x(t),u(t)\right),\\
\dot{y}(t)=0,
\end{cases}\\
\text{and to the initial conditions }
x(a)=\alpha \text{ and } y(a)= \frac{\phi(x(b))}{b-a}.
\end{gathered}
\end{equation*}

Using Pontryagin's maximum principle,  
the following result follows \cite{MyArt05}.

\begin{theorem}[Higher-order delayed Euler--Lagrange equations 
and transversality conditions \cite{MyArt05}]
\label{thm:E-L}
If $(x(\cdot),z(\cdot))$ is an extremizer to problem $\mathbf{(H_\tau^n)}$
that satisfies the conditions $x^{(k)}(t)=\mu^{(k)}(t)$, with $\mu \in PC^n ([a-\tau,a];
\mathbb{R}^m)$, $k=0,\dots,n-1$ and $t\in [a-\tau, a]$, 
then the following two Euler--Lagrange equations hold:
\begin{equation*}
\sum_{l=0}^{n}(-1)^l\frac{d^l}{dt^l}\left(\psi_z(t)\frac{\partial L}{\partial
x^{(l)}}[x;z]^n_\tau(t)+\psi_z(t+\tau)\frac{\partial L}{\partial
x_\tau^{(l)}}[x;z]^n_\tau(t+\tau)\right)=0,
\end{equation*}
for $t \in [a,b-\tau]$, and
\begin{equation*}
\sum_{l=0}^{n}(-1)^l\frac{d^l}{dt^l}\left(\psi_z(t)
\frac{\partial L}{\partial x^{(l)}}[x;z]^n_\tau(t)\right)=0,
\end{equation*}
for $t \in [b-\tau,b]$, where $\psi_z$ is defined by
\begin{equation*}
\psi_z(t)=e^{\int_t^b\frac{\partial L}{\partial z}[x;z]^n_\tau(\theta)d\theta},
\quad t \in [a,b].
\end{equation*}
Furthermore, the following transversality conditions are satisfied:
\begin{equation*}
\sum_{l=0}^{n-k}(-1)^l\frac{d^l}{dt^l}\left(\psi_z(t)
\frac{\partial L}{\partial x^{(l+k)}}[x;z]^n_\tau(t)\right)\bigg\vert_{t=b}=0, \quad k=1,\dots,n.
\end{equation*}
\end{theorem}

In addition to previous result, we were able to obtain in \cite{MyArt05} 
expressions for the multipliers related to $z$ and $x$, and also the 
expression of the Hamiltonian of problem $\mathbf{(H_\tau^n)}$. They are, 
respectively:
\begin{equation}
\label{psi_z}
\psi_z(t)=e^{\int_t^b\frac{\partial L}{\partial z}[x;z]^n_\tau(\theta)d\theta},
\quad t \in [a,b],
\end{equation}
\begin{equation}
\label{phi_k}
\begin{split}
&\phi_{k}(t)=\sum_{l=0}^{n-k}(-1)^{l+1}\frac{d^l}{dt^l}\left(\psi_{z}(t+\tau)
\frac{\partial L}{\partial x_\tau^{(l+k)}}[x;z]^n_\tau(t+\tau)\right), \quad t \in [a-\tau,a],\\
&\phi_{k}(t)=\sum_{l=0}^{n-k}(-1)^{l+1}\frac{d^l}{dt^l}\left(\psi_z(t)
\frac{\partial L}{\partial x^{(l+k)}}[x;z]^n_\tau(t)
+\psi_{z}(t+\tau)
\frac{\partial L}{\partial x_\tau^{(l+k)}}[x;z]^n_\tau(t+\tau)\right), 
\end{split}
\end{equation}
$t\in [a,b]$, and
\begin{equation}
\label{Hamilt H-O delay}
H=\sum_{k=1}^n\phi_k(t) \cdot x^{(k)}(t)
+\psi_z(t) L[x;z]^n_\tau(t), \ t\in [a,b].
\end{equation}
Before presenting Noether's second theorem for the optimal control \eqref{problem P}, 
we need to introduced a notion of invariance. In this paper we follow the definition 
of semi-invariance presented in \cite{Torres2003MR1980565}.

\begin{definition}[Semi-invariance of problem \eqref{problem P} 
under a group of symmetries \cite{Torres2003MR1980565}]
\label{DEF inv symm}
Let $p:[a,b]\rightarrow \mathbb{R}^d$ be an arbitrary 
function of class $C^q$. Using the notation
$$
\alpha(t):=\left(t,x(t),u(t),p(t),\dot{p}(t),\dots,p^{(q)}(t)\right),
$$
we say that the optimal control problem \eqref{problem P} 
is semi-invariant if there exist a $C^1$ transformation group
\begin{equation}
\label{eq:trf:g}
\begin{gathered}
g:[a,b]\times\mathbb{R}^m\times \Omega\times\mathbb{R}^{d\times(q+1)}
\rightarrow\mathbb{R}\times\mathbb{R}^m\times \mathbb{R}^r,\\
g(\alpha(t))=\left(\mathsf{T}(\alpha(t)),\mathsf{X}(\alpha(t)),\mathsf{U}(\alpha(t))\right),
\end{gathered}
\end{equation}
which for $p(t)=\dot{p}(t)=\cdots=p^{(q)}(t)=0$ coincides with the identity transformation 
for all $(t,x,u)\in[a,b]\times \mathbb{R}^m\times \Omega$, satisfying the following conditions:
\begin{multline*}
\left(\theta_0\cdot p(t)+\theta_1\cdot \dot{p}(t)
+\cdots+\theta_q\cdot p^{(q)}(t)\right)
\frac{d}{dt}f(t,x(t),u(t))+f(t,x(t),u(t))\\
+\frac{\phi(x(b))}{b-a} + \frac{d}{dt}F(\alpha(t))
=\left(f(g(\alpha(t)))+\frac{\phi(X\left(\alpha(b)\right))}{T(\alpha(b))-T(\alpha(a))}\right)
\frac{d}{dt}\mathsf{T}(\alpha(t)),
\end{multline*}
\begin{equation*}
\frac{d}{dt}\mathsf{X}(\alpha(t))=\varphi\left(g(\alpha(t))\right)\frac{d}{dt}\mathsf{T}(\alpha(t)),
\end{equation*}
for some function $F$ of class $C^1$ and some $\theta_0,\dots, \theta_q \in \mathbb{R}^d$.
\end{definition}

\begin{remark}
The group of transformations $g$ \eqref{eq:trf:g} is usually called 
a gauge symmetry of the optimal control problem, in order to emphasize 
the fact that the transformations depend on arbitrary functions and, 
therefore, have local nature.
\end{remark}

\begin{theorem}[Noether's second theorem for the optimal control 
problem \eqref{problem P} \cite{Torres2003MR1980565}]
\label{thm 2_Noether OC}
If problem \eqref{problem P} is semi-invariant under a group of symmetries 
as in Definition~\ref{DEF inv symm}, then there are 
$d(q+1)$ Noether currents of the form
\begin{multline*}
\left.\frac{\partial F(\alpha(t))}{\partial p_J^{(I)}}\right|_0
+\theta_J^I\left(f(t,x(t),u(t)+\frac{\phi(x(b))}{b-a}\right) \\
+ \psi(t)\cdot \left. \frac{\partial \mathsf{X}(\alpha(t))}{\partial p_J^{(I)}}\right|_0
-\left.H(t,x(t),u(t),\psi(t))\frac{\partial \mathsf{T}(\alpha(t))}{\partial p_J^{(I)}}\right|_0
\end{multline*}
for $I=0,\dots,q,$ $J=1,\dots,d,$ where
$H$ is defined in \eqref{eq:def:Hamiltonian} and $\left.(*)\right|_0$ 
stands for $\left.(*)\right|_{p(t)=\dot{p}(t)=\dots=p^{(q)}(t)=0}$.
\end{theorem}

\begin{remark} 
It is clear that if $\varphi=u$, $\theta_0=\dots=\theta_q=0$ and $F\equiv0$, 
and the transformation group $g$ does not depend on the derivatives 
of the state variables, then Theorem~\ref{thm 2_Noether OC} reduces 
to the classical Noether's second theorem for the basic problem 
of the calculus of variations.
\end{remark}


\section{Proof of main result}
\label{sec:MainRes}

The central ideia of the proof of our main result, 
Noether's second theorem for the higher-order 
variational problem of Herglotz type with time delay, 
is to rewrite problem $\mathbf{(H_\tau^n)}$ as a 
non-delayed optimal control problem. For this, 
we assume, without loss of generality, 
that the initial time is zero ($a=0$)
and the final time $b$ is an integer multiple of $\tau$, that is,
$b=N\tau$ for some $N \in \mathbb{N}$ (see Remark~\ref{rem:red:bN}).
Therefore, we can divide the interval $[a,b]$ into $N$ equal parts.  
Fix $t \in [0,\tau]$ and introduce variables $x^{k;i}$ and $z_j$ 
with $k=0,\dots,n$, $i=0,\dots,N$, and $j=1,\dots, N+1$, as follows:
\begin{equation}
\label{reduction}
\begin{gathered}
x^{k;i}(t)=x^{(k)}(t+(i-1)\tau),
\quad z_j(t)=z(t+(j-1)\tau),\\
\dot{z}_j(t)=L_j(t),\quad
x^{k;N+1}(t)=0,
\quad \dot{z}_{N+1}(t)=L_{N+1}=0
\end{gathered}
\end{equation}
with
$$
L_j(t):=L\left(t+(j-1)\tau,x^{0;j}(t),\dots,x^{n;j}(t),
x^{0;j-1}(t),\dots,x^{n;j-1}(t),z_j(t)\right).
$$
Note that the index $k$ is related to the order
of the derivative of $x$, $i$ is related to the $i$th subinterval 
of $[-\tau,N\tau]$, and $j$ is related to the $j$th subinterval
of $[0,(N+1)\tau]$. Consequently, the higher-order problem 
of Herglotz with time delay $\mathbf{(H_\tau^n)}$ can be written 
as a first-order optimal control problem without time delay as follows:
\begin{equation}
\label{Problem:OC:non:delayed:form}
\begin{gathered}
z_N(\tau)\longrightarrow \mathrm{extr},\quad \text{subject to }\\
\begin{cases}
\dot{x}^{k;i}(t)=x^{k+1;i}(t),  \\
x^{k;N+1}(t)=0,\\
\dot{z}_j(t)=L_j(t), \\
\dot{z}_{N+1}(t)=L_{N+1}(t)=0
\end{cases}\\
\text{ for all } t \in[0,\tau] \text{ and with the initial conditions}\\
x^{k;0}(0)=\mu^{(k)}(-\tau),\quad
x^{k;i}(0)=x^{k;i-1}(\tau),\\
z_1(0)=\gamma,\quad \gamma \in \mathbb{R},
\quad z_j(0)=z_{j-1}(\tau)
\end{gathered}	
\end{equation}
for $k=0, \dots, n-1$, $i=0,\dots, N$ and $j=1,\dots, N$.
In this form, we look to $x^{k;i}$ and $z_j$ as state variables
and to $u_i:=x^{n;i}$ as the control variables.

\begin{remark}
\label{rem:red:bN}
In our previous reduction, we considered the simplest case where $b=N\tau$.
If $b$ is not an integer multiple of $\tau$, then there is an integer $N$
such that $(N-1)\tau<b<N\tau$. In that case, the only modification required
in the change of variables indicated in \eqref{reduction} is to consider the 
variables $x^{k;N}$, $k=0, \dots, n$, and $\dot{z}_N$ as defined in \eqref{reduction}
for $t \in [0,b-(N-1)\tau]$ and zero for $t\in ]b-(N-1)\tau, \tau]$. Note that with this
minor change, the function to be extremized remains the same and, therefore,  
we can consider that $b=N\tau$.
\end{remark}

\begin{remark}[Semi-invariance of problem $\mathbf{(H_\tau^n)}$ under a group of symmetries]
If there is a $C^1$ transformation group
\begin{equation}
\label{eq:g}
\begin{gathered}
g:[a,b]\times\mathbb{R}^{2m(n+1)+1}\times\mathbb{R}^{d (q+1)}
\rightarrow\mathbb{R}\times\mathbb{R}^m\times \mathbb{R},\\
g(\alpha(t))=\left(\mathsf{T}(\alpha(t)),\mathsf{X}(\alpha(t)),\mathsf{Z}(\alpha(t))\right),
\end{gathered}
\end{equation}
where $\alpha(t)$ stands for
\begin{equation*}
\left(t,x(t),\dot{x}(t),\dots,x^{(n)}(t),x(t-\tau),\dot{x}(t-\tau),
\dots,x^{(n)}(t-\tau),z(t),p(t),\dot{p}(t),\dots,p^{(q)}(t)\right),
\end{equation*}
which for $p(t)=\dot{p}(t)=\cdots=p^{(q)}(t)=0$ coincides 
with the identity transformation for all $(t,x,z)\in[a-\tau,b]
\times \mathbb{R}^m\times \mathbb{R}$, and such that problem 
$\mathbf{(H_\tau^n)}$ satisfies the two equations
\begin{equation}
\label{invariance H_tau^n EQ1}
\frac{z(b)}{b-a} + \frac{d}{dt}F(\alpha(t))
= \frac{Z(\alpha(b))}{T(\alpha(b))-T(\alpha(a))}\,\frac{d}{dt}\mathsf{T}(\alpha(t))
\end{equation}
and
\begin{equation}
\label{invariance H_tau^n EQ2}
\frac{d}{dt}\mathsf{Z}(\alpha(t))=L(g(\alpha(t)))\frac{d}{dt}\mathsf{T}(\alpha(t))
\end{equation}
for some function $F$ of class $C^1$, where
\begin{equation*}
\begin{split}
\frac{d}{d\mathsf{T}}\mathsf{X}(\alpha(t))
&=\frac{\frac{d}{dt}\mathsf{X}(\alpha(t))}{
\frac{d}{dt}\mathsf{T}(\alpha(t))} \text{ and }
\frac{d^k}{d\mathsf{T}^k}\mathsf{X}(\alpha(t))
=\frac{\frac{d}{dt}\left(\frac{d^{k-1}}{d\mathsf{T}^{k-1}}
\mathsf{X}(\alpha(t))\right)}{\frac{d}{dt}\mathsf{T}(\alpha(t))},
\end{split}
\end{equation*}
$k=2,\ldots,n$, then problem $\mathbf{(H_\tau^n)}$ is semi-invariant 
under a group of symmetries as in Definition~\ref{DEF inv symm}.
\end{remark}

We are now in a position to formulate and prove our main result.

\begin{theorem}[Noether's second  theorem for problem $\mathbf{(H_\tau^n)}$]
\label{thm 2_Noether}
If problem $\mathbf{(H_\tau^n)}$ is semi-invariant under a group 
of symmetries \eqref{eq:g}, that is, if  
\eqref{invariance H_tau^n EQ1}--\eqref{invariance H_tau^n EQ2}
holds, then there are $d(q+1)$ Noether currents of the form
\begin{multline*}
\left.\frac{\partial F(\alpha(t))}{\partial p_J^{(I)}}\right|_0+\theta_J^I\frac{z(b)}{b-a}\\
+ \sum_{k=1}^{n}\phi_k(t)\cdot \left. \frac{\partial }{\partial p_J^{(I)}}\left(\frac{d^{k-1}}{d\mathsf{T}^{k-1}}\mathsf{X}(\alpha(t))\right)\right|_0
+\psi_z(t)\cdot \left. \frac{\partial \mathsf{Z}(\alpha(t))}{\partial p_J^{(I)}}\right|_0\\
\left. -H(t,x(t),\dot{x}(t),\dots, x^{(n)}(t),z(t),\phi_1(t),
\dots,\phi_n(t),\psi_z(t))\frac{\partial \mathsf{T}(\alpha(t))}{\partial p_J^{(I)}}\right|_0,
\end{multline*}
$t \in [a,b]$, for $I=0,\dots,q,$ $J=1,\dots,d$, and $\theta^I_J \in \mathbb{R}^d$, 
where $\psi_z, \phi_k$ and $H$ are defined, respectively, in \eqref{psi_z}--\eqref{Hamilt H-O delay} 
and $\left.(*)\right|_0$ stands for $\left.(*)\right|_{p(t)=\dot{p}(t)=\dots=p^{(q)}(t)=0}$.
\end{theorem}

\begin{proof}
In order to prove the result, we start by considering problem 
$\mathbf{(H_\tau^n)}$ in its optimal control and non-delayed 
form \eqref{Problem:OC:non:delayed:form}. First, we prove that 
if $\mathbf{(H_\tau^n)}$ is semi-invariant under a group of symmetries, 
that is, if there exists a $C^1$ transformation group \eqref{eq:g} 
satisfying \eqref{invariance H_tau^n EQ1}--\eqref{invariance H_tau^n EQ2},
then the non-delayed optimal control problem 
\eqref{Problem:OC:non:delayed:form} is invariant in the sense 
of Definition~\ref{DEF inv symm}. Observe that 
\eqref{invariance H_tau^n EQ1} is equivalent to
\begin{equation}
\label{inv eq1}
\frac{z_N(\tau)}{\tau}+\frac{d}{dt} \tilde{F}(\alpha(t))
= \frac{Z_N(\alpha(\tau))}{T\left(\alpha(\tau)\right)}\frac{d}{dt}\mathsf{T}(\alpha(t)),
\end{equation}
where $\tilde{F}$ is defined for all 
$t \in [0,\tau]$ by $\tilde{F}(\alpha)(t):=N \cdot F(\alpha)(t)$.
Now, defining 
\begin{equation*}
\begin{split}
\mathsf{X}_{k;i}(\alpha(t))&:=\frac{d^k }{d\mathsf{T}^k}\mathsf{X}(\alpha(t+(i-1)\tau)),\\
\mathsf{T}_i(\alpha(t))&:=\mathsf{T}(\alpha(t+(i-1)\tau)),\\
\mathsf{Z}_j(\alpha(t))&:=\mathsf{Z}(\alpha(t+(j-1)\tau))
\end{split}
\end{equation*}
for fixed $t \in [0, \tau]$, we have
\begin{equation}
\label{inv eq 2.1}
\frac{d}{dt}\mathsf{X}_{k;i}(\alpha(t))
=\mathsf{X}_{k+1;i}(\alpha(t))\frac{d}{dt}\mathsf{T}_{i}(\alpha(t))
\end{equation}
and
\begin{equation}
\label{inv eq 2.2}
\frac{d }{dt}\mathsf{Z}_j(\alpha(t))
= L_j\left(g(\alpha(t))\right) \frac{d}{dt}\mathsf{T}_j(\alpha(t)),
\end{equation}
for $k=0,\dots, n-1$, $i=0,\dots N$, and $j=1,\dots,N$.
From \eqref{inv eq1}--\eqref{inv eq 2.2}, we conclude 
that the non-delayed optimal control problem \eqref{Problem:OC:non:delayed:form} 
is semi-invariant in the sense of Definition~\ref{DEF inv symm}. 
This kind of semi-invariance is the required condition 
for application of the second Noether 
theorem for optimal control (Theorem~\ref{thm 2_Noether OC}), 
which asserts the existence of $d(q+1)$ Noether currents of the form
\begin{equation*}
\begin{split}
\left.\frac{\partial F(\alpha(t))}{\partial p_J^{(I)}}\right|_0
+\theta_J^I\frac{z_N(\tau)}{\tau}
+\sum_{k=1}^{n}\sum_{i=0}^{N}\phi_{k;i}(t)
\cdot \left. \frac{\partial \mathsf{X}_{k-1;i}(\alpha(t))}{\partial p_J^{(I)}}\right|_0
+\sum_{j=1}^{N}\psi_j(t)\cdot \left. \frac{\partial \mathsf{Z}_j(\alpha(t))}{\partial p_J^{(I)}}\right|_0\\
- \left[ \sum_{k=1}^{n}\sum_{i=0}^{N}\phi_{k;i}(t)\cdot x^{k;i}(t)
+\sum_{j=1}^{N} \psi_j(t) L_j(t)\right]
\left. \frac{\partial \mathsf{T}(\alpha(t))}{\partial p_J^{(I)}}\right|_0,
\end{split}
\end{equation*}
$t \in [0,\tau]$, for $I=0,\dots,q,$ $J=1,\dots,d$, where 
$\phi_{k;i}$ and $\psi_j$ are defined from \eqref{psi_z}--\eqref{phi_k}:
\begin{equation*}
\phi_{k;i}(t)=\phi_k(t+(i-1)\tau) \text{ and } \psi_j(t)=\psi_z(t+(i-1)\tau),
\end{equation*}
for $i=0,\dots,N$ and $j=1,\dots, N$. Finally, we rewrite the result 
in the original variables, obtaining that there are $d(q+1)$ Noether currents of the form
\begin{multline*}
\left.\frac{\partial F(\alpha(t))}{\partial p_J^{(I)}}\right|_0+\theta_J^I\frac{z(b)}{b-a}
+\sum_{k=1}^{n}\phi_k(t)\cdot \left. \frac{\partial \mathsf{X}_k(\alpha(t))}{\partial p_J^{(I)}}\right|_0
+\psi_z(t)\cdot \left. \frac{\partial \mathsf{Z}(\alpha(t))}{\partial p_J^{(I)}}\right|_0\\
-\left.	H(t,x(t),\dot{x}(t),\dots, x^{(n)}(t),z(t),\phi_1(t),
\dots,\phi_n(t),\psi_z(t))\frac{\partial \mathsf{T}(\alpha(t))}{\partial p_J^{(I)}}\right|_0.
\end{multline*}
This concludes the proof.
\end{proof}

Our result is new even for first-order generalized variational problems.

\begin{corollary}
\label{coroll:n=1}
If the first-order problem of Herglotz with time delay
\begin{equation*}
\begin{gathered}
z(b)\longrightarrow \mathrm{extr},\\
\dot{z}(t)=L\left(t,x(t),\dot{x}(t),x(t-\tau),\dot{x}(t-\tau),z(t)\right),
\quad t \in [a,b], \\
z(a)=\gamma \in \mathbb{R}, \quad x(t)=\mu(t), \quad t \in [a-\tau,a],
\end{gathered}
\end{equation*}
where $\mu$ is a given piecewise initial function, is semi-invariant, 
then there exist $d(q+1)$ Noether currents of the form
\begin{multline*}
\left.\frac{\partial F(\alpha(t))}{\partial p_J^{(I)}}\right|_0+\theta_J^I\frac{z(b)}{b-a}
+\phi_1(t)\cdot \left. \frac{\partial \mathsf{X}(\alpha(t))}{\partial p_J^{(I)}}\right|_0
+\psi_z(t)\cdot \left. \frac{\partial \mathsf{Z}(\alpha(t))}{\partial p_J^{(I)}}\right|_0\\
-\left.	\left[\phi_1(t)\dot{x}(t)+\psi_z(t)L[x;z]_\tau^1(t)\right]
\frac{\partial \mathsf{T}(\alpha(t))}{\partial p_J^{(I)}}\right|_0,
\end{multline*}
$t \in [a,b]$, for $I=0,\dots,q,$ $J=1,\dots,d$, where $\phi_1$ 
is given by \eqref{phi_k} and $\psi_z$ by \eqref{psi_z}.
\end{corollary}

\begin{proof}
Consider Theorem~\ref{thm 2_Noether} with $n=1$.
\end{proof}

As a corollary of Corollary~\ref{coroll:n=1}, we obtain a new result 
for delayed classical problems of the Calculus of Variations.

\begin{corollary}
\label{coroll:2}
If the first-order variational problem with time delay 
\begin{equation*}
\int_a^b L(t,x(t),\dot{x}(t),x(t-\tau),\dot{x}(t-\tau))dt \longrightarrow \mathrm{extr},
\end{equation*}
with $x(t)=\mu(t)$, $t \in [a-\tau,a]$, for a given piecewise initial function $\mu$, 
is semi-invariant, then there exists $d(q+1)$ Noether currents of the form
\begin{multline*}
\left.\frac{\partial F(\alpha(t))}{\partial p_J^{(I)}}\right|_0
+\phi_1(t)\cdot \left. \frac{\partial \mathsf{X}(\alpha(t))}{\partial p_J^{(I)}}\right|_0
+\theta^{I}_{J}\frac{z(b)}{b-a}\\
-\left.	\Big[\phi_1(t)\dot{x}(t)+L\left(t,x(t),\dot{x}(t),x(t-\tau),\dot{x}(t-\tau)\right)\Big]
\frac{\partial \mathsf{T}(\alpha(t))}{\partial p_J^{(I)}}\right|_0,
\end{multline*}
$t \in [a,b]$, for $I=0,\dots,q,$ $J=1,\dots,d$, 
where $\phi_1$ is given by \eqref{phi_k}.
\end{corollary}

\begin{proof}
Consider Corollary~\ref{coroll:n=1} with $L$ not depending on $z$.
\end{proof}


\section{Example}
\label{sec:ex}

In order to illustrate our results, we present a simple example 
that cannot be covered using available results in the literature. 
Consider an arbitrary interval $[a,b]$ and let $\tau \in \mathbb{R}$ 
be a nonnegative real number such that $\tau <b-a$. We address 
the following problem with $m=d=q=1$:
\begin{equation}
\label{eq:prb:ex}
\begin{gathered}
z(b)\rightarrow\mathrm{extr},\\
\dot{z}(t)=x(t-\tau) z(t), \quad t\in[a,b], \\
\text{subject to } z(a)=\gamma, \quad x(t) = \mu(t), \quad t\in[a-\tau,a],
\end{gathered}
\end{equation}
where $\mu \in PC^1([a-\tau,a];\mathbb{R})$ is a given initial function.
Let $p$ be a $C^1([a,b];\mathbb{R})$ function and consider the $C^1$ group of symmetries
\begin{equation*}
g(\alpha(t))=\left(t+p(t),\frac{x(t-\tau)}{1+\dot{p}(t)},z(t)\right),
\end{equation*}
that is, 
\begin{equation*}
\begin{aligned}
&\mathsf{T}(\alpha(t))=T\left(t,p(t)\right) = t + p(t),\\
&\mathsf{X}(\alpha(t))=X\left(x(t-\tau),\dot{p}(t)\right) = \frac{x(t-\tau)}{1+\dot{p}(t)},\\ 
&\mathsf{Z}(\alpha(t))=Z\left(z(t)\right) = z(t),
\end{aligned}
\end{equation*}
which for $p(t) = \dot{p}(t) = 0$, $t\in [a,b]$, reduce to the identity transformations.
Observe that the problem under study is semi-invariant. Indeed, 
\eqref{invariance H_tau^n EQ1} is verified with 
$$
F(t)=\frac{z(b)}{b-a+p(b)-p(a)} \left(t + p(t)\right) -\frac{z(b)}{b-a} t
$$
and \eqref{invariance H_tau^n EQ2} is also valid because
\begin{equation*}
\frac{d}{dt}\mathsf{Z}(\alpha(t))=\dot{z}(t)
=\frac{x(t-\tau)}{1+\dot{p}(t)} z(t) (1+\dot{p}(t)) 
=L(g(\alpha(t)))\frac{d}{dt}\mathsf{T}(\alpha(t)).
\end{equation*}
From Theorem~\ref{thm 2_Noether}, we have that there are two Noether currents of the form
\begin{multline*}
\left.\frac{\partial F(\alpha(t))}{\partial p^{(I)}}\right|_0+\theta^I\frac{z(b)}{b-a}
+\phi_1(t)\cdot \left. \frac{\partial \mathsf{X}(\alpha(t))}{\partial p^{(I)}}\right|_0
+\psi_z(t)\cdot \left. \frac{\partial \mathsf{Z}(\alpha(t))}{\partial p^{(I)}}\right|_0\\
-\left.	\left[\phi_1(t)\dot{x}(t)+\psi_z(t)L[x;z]_\tau^1(t)\right]
\frac{\partial \mathsf{T}(\alpha(t))}{\partial p^{(I)}}\right|_0, \quad I=0,1.
\end{multline*}
Noting that $\phi_1(t)=0$ and $\psi_z(t)=e^{\int_t^b x(s-\tau) ds}$, $t\in [a,b]$, 
the second Noether current reduces to a constant 
while the first gives a nontrivial conclusion: it asserts that
\begin{equation*}
x(t-\tau) z(t) e^{\int_t^b x(s-\tau) ds} 
\end{equation*}
is constant along the extremals of problem \eqref{eq:prb:ex}.


\section{Concluding remarks}
\label{sec:conc}

We have deduced new necessary conditions for higher-order generalized 
variational problems with time delay that are semi-invariant under 
a group of transformations that depends on arbitrary functions. 
The conditions are potentially useful, because for many variational problems, 
the Euler--Lagrange equations and transversality conditions are not enough 
to obtain an explicit solution. Our main result is new even 
for classical delayed variational problems.


\section*{Acknowledgements}

This research is part of first author's Ph.D. project,
which is carried out at University of Aveiro.
It was partially supported by Portuguese funds through
the Center for Research and Development in Mathematics
and Applications (CIDMA) and the Portuguese Foundation
for Science and Technology (FCT), within project 
UID/MAT/04106/2013. The authors are grateful to an 
anonymous Reviewer for several comments and suggestions, 
which showed them where to clarify the paper 
and how to improve its quality.



\medskip
Received September 2016; revised March 2017.
\medskip


\end{document}